\documentclass[a4paper,10pt]{article}

\usepackage{latexsym}
\usepackage{amsfonts}
\usepackage[latin1]{inputenc}
\usepackage{amsmath}
\usepackage{amssymb}
\usepackage[all]{xy}
\usepackage[normalem]{ulem}
\usepackage{lscape}
\usepackage{color}

\usepackage{times,amsmath} 

\usepackage{adjustbox}

\newtheorem{remark}{Remark}[section]
\newtheorem{example}{Example}[section]
\newtheorem{definition}{Definition}[section]
\newtheorem{proposition}{Proposition}[section]
\newtheorem{corollary}{Corollary}[section]
\newtheorem{lemma}{Lemma}[section]
\newtheorem{theorem}{Theorem}[section]
\newenvironment{proof}[1]{\noindent \textsc{Proof:} #1}{\hfill \rule{2mm}{2mm}}

\begin{document}

\begin{center}
 {\large Ordinal Sums of Fuzzy Negations: Main Classes and Natural Negations}
\end{center}

\begin{center}
Annaxsuel A. de Lima$^{a,b}$, Benjam\'in Bedregal$^b$, Ivan Mezzomo$^c$
\end{center}

\footnotesize       
\begin{center}
 $^a$ Instituto Federal de Educa\c{c}\~ao, Ci\^encia e Tecnologia do Rio Grande do Norte -- IFRN \\ Campus S\~ao Paulo do Potengi, Rio Grande do Norte, Brazil. E-mail: annaxsuel.araujo\@ ifrn.edu.br  \\
$^b$ Universidade Federal do Rio Grande do Norte -- UFRN \\ Programa de P\'os-Gradua\c{c}\~ao em Sistemas e Computa\c{c}\~ao -- PPgSC \\ Departamento de Inform\'atica e Matem\'atica Aplicada  -- DIMAp \\ Natal, Rio Grande do Norte, Brazil. E-mail: bedregal\@ dimap.ufrn.br \\
$^c$ Universidade Federal Rural de Semi-\'Arido -- UFERSA \\ Departamento de Ci\^encias Naturais, Matem\'atica e Estat\'istica -- DCME \\ Mossor\'o, Rio Grande do Norte, Brazil. E-mail: imezzomo\@ ufersa.edu.br
\end{center}

\normalsize


\begin{abstract}
In the context of fuzzy logic, ordinal sums provide a method for constructing new functions from existing functions, which can be triangular norms, 
triangular conorms, fuzzy negations, copulas, overlaps, uninorms, fuzzy implications, among others. As our main contribution, we establish conditions 
for the ordinal sum of a family of fuzzy negations to be a fuzzy negation of a specific class, such as strong, strict, continuous, invertible and 
frontier. Also, we relate the natural negation of the ordinal sum on families of t-norms, t-conorms and fuzzy implications with the ordinal sum of the 
natural negations of the respective families of t-norms, t-conorms and fuzzy implications. This motivated us to introduces a new kind of ordinal sum for families of fuzzy implications.
\end{abstract}


\section{Introduction}

The concept of the fuzzy set was introduced by Zadeh (1965) and, since then, several mathematical concepts such as number, group, topology, differential 
equation, and so on, have been fuzzified. There are several ways to extend the propositional connectives for a set $[0,1]$, but in general these 
extensions do not preserves all the properties of the classical logical connectives. Triangular norms (t-norms) and triangular conorms (t-conorms) were first 
studied by Menger \cite{menger} and also by Schweizer and Sklar \cite{sklar} in probabilistic metric spaces and they are used to represent the logical 
conjunction in fuzzy logic and the interception of fuzzy sets, whereas t-conorms are used to represent the logical disjunction in fuzzy logic and the union 
in fuzzy set theory.

In 1965, L. A. Zadeh introduced the notion of fuzzy negation in \cite{Zadeh65}, known as standard negation, in order to represent the logical negation and 
the complement of fuzzy sets. Since then, several important classes of fuzzy negations have been proposed with different motivations, as we can see in 
\cite{est, FKL99, hig, lowen, ov, trillas}. Fuzzy negations have applications in several areas, such as decision making, stock investment, computing with
words, mathematical morphology and associative memory,  as the presented in \cite{bedregal12,bent,FKL99,GMM16,VS08,zhao}.

The ordinal sums construction was first introduced, in the context of semigroups, by Climescu in \cite{Cli46} and Clifford in \cite{Cli54}.
In the context of fuzzy logic, the ordinal sums were first studied  for triangular norms and triangular conorms in \cite{SS63} in order to provide a method to 
construct new t-norms and t-conorms from other t-norms and t-conorms (for more details see \cite{klement}). However, the ordinal sums of several others 
important fuzzy connectives also has been studied, such as, for example, the ordinal sums of copulas \cite{Nel06}, overlap functions \cite{DB14}, 
uninorms \cite{MeZ16,MeZ17}, fuzzy implications \cite{DK16,Su15} and fuzzy negations \cite{BSJIFS}. In particular, the ordinal sums of fuzzy negations 
proposed in \cite{BSJIFS} were made in the context of Morgan's triples and they were not deeply studied.

In this paper, we consider the notion of the ordinal sums of a family of fuzzy negations, as introduced in \cite{BSJIFS}, and prove some results involving these 
concepts. In particular, we establish conditions for the ordinal sum of a family of fuzzy negations resulting in a fuzzy negation belonging to a class of fuzzy 
negations, such as strict, strong, frontier, continuous and invertible.

This paper is organized as follows: Section 2 provides a review of concepts such as t-norms, t-conorms, fuzzy implications, fuzzy negations, 
natural fuzzy negations, ordinal sums of a family of t-norms, t-conorms and fuzzy implications. In Section 3, we prove that the ordinal sum of a family 
of fuzzy negations is a fuzzy negation and we prove results involving concepts of ordinal sums of a family of fuzzy negations and equilibrium point. 
In Section 4, we establish conditions for the ordinal sum of a family of fuzzy negations resulting in a fuzzy negation belonging to a class of fuzzy 
negations such as strict, strong, frontier, continuous and invertible. In Section 5, we define the left ordinal sum of a family of fuzzy implications 
and prove results involving ordinal sums of a family of t-norms, t-conorms and fuzzy implications. Also, we prove that the natural negation of the left 
ordinal sum of a family of fuzzy implications is the same to the ordinal sum of a family of fuzzy negations. Finally, Section 6 contains the final 
considerations and future works.

\section{Preliminaries}\label{sect2}

In this section, we will briefly review some basic concepts  which are necessary for the development of this paper. The definitions and additional results can be found in \cite{alsina, atanassov99, bedregal10, bedregal12, bustince99, bustince03, FodorRoubens1994, klement}.

\subsection{t-norms, t-conorms, fuzzy implications  and fuzzy negations}

\begin{definition}
A function $T:[0,1]^{2}\rightarrow [0,1]$ is a t-norm if, for all $x,y,z\in [0,1]$, the following axioms are satisfied:

1. Symmetry: $T(x,y)=T(y,x)$;

2. Associativity: $T(x, T(y,z))=T(T(x,y),z)$;

3. Monotonicity: If $x\leq y$, then $T(x,z)\leq T(y,z)$;

4. One identity: $T(x,1)=x$.\\
\end{definition}

A t-norm $T$ is called positive if it satifies the condition: $T(x, y)=0$ iff $x=0$ or $y=0$.\\

\begin{example} 
Some examples of t-norms:

1. G\"odel t-norm: $T_G(x, y)=\min(x, y)$;

2. Product t-norm: $T_P(x, y)=x \cdot y$;

3. \L ukasiewicz t-norm: $T_L(x, y)= \max(0, x+y-1)$;

4. Drastic t-norm: \begin{eqnarray*}
T_D(x, y)=\left\{
\begin{array}{ll}
0 & \mbox{if} ~ (x, y)\in [0,1[^2; \\
\min(x, y) & \mbox{otherwise.}
\end{array}
\right.
\end{eqnarray*}
\hfill\rule{2mm}{2mm}
\end{example}

\begin{definition}
A function $S:[0,1]^{2}\rightarrow [0,1]$ is a t-conorm if, for all $x,y,z\in [0,1]$, the following axioms are satisfied:

1. Symmetry: $S(x,y)=S(y,x)$;

2. Associativity: $S(x, S(y,z))=S(S(x,y),z)$;

3. Monotonicity: If $x\leq y$, then $S(z, x)\leq S(z, y)$;

4. Zero identity: $S(x,0)=x$.\\
\end{definition}

A t-conorm $S$ is called positive if  it satifies the condition: $S(x, y)=1$ iff $x=1$ or $y=1$.\\
\begin{example} 
Some examples of t-conorms:

1. G\"odel t-conorm: $S_G(x, y)=\max(x, y)$;

2. Probabilistic sum: $S_P(x, y)=x+y-x \cdot y$;

3. \L ukasiewicz t-conorm: $S_L(x, y)= \min(x+y, 1)$;

4. Drastic sum: \begin{eqnarray*}
S_D(x, y)=\left\{
\begin{array}{ll}
1 & \mbox{if} ~ (x, y)\in ]0,1]^2; \\
\max(x, y) & \mbox{otherwise.}
\end{array}
\right.
\end{eqnarray*}
\hfill\rule{2mm}{2mm}
\end{example}

\begin{definition}\cite[Definition 1.1.1]{BJ08}
A function $J:[0,1]^2\rightarrow [0,1]$ is called a fuzzy implication if it is satisfies the following conditions:

J1: $J$ is non-increasing with respect to the first variable;

J2: $J$ is non-decreasing with respect to the second variable;

J3: $J(0, 0) = 1$;

J4: $J(1, 1) = 1$;

J5: $J(1, 0) = 0$.
\end{definition}

\begin{example} 
Some examples of fuzzy implications:\\
1. G\"odel implication:  \begin{eqnarray*}
J_G(x, y)=\left\{
\begin{array}{ll}
1 & \mbox{if} ~ x \leq y; \\
y & \mbox{otherwise.}
\end{array}
\right.
\end{eqnarray*}
2. Rescher implication:  \begin{eqnarray*}
J_{RS}(x, y)=\left\{
\begin{array}{ll}
1 & \mbox{if} ~ x \leq y; \\
0 & \mbox{otherwise.}
\end{array}
\right.
\end{eqnarray*}
3. Kleene-Dienes implication: $J_{KD}(x, y)= \max(1-x, y)$.\\
\end{example}

A function $N:[0,1]\rightarrow [0,1]$ is a fuzzy negation if

 N1: $N(0)=1$ and $N(1)=0$;

 N2:  If $x\leq y$, then $N(x)\geq N(y)$, for all $x,y\in [0,1]$.

A fuzzy negations $N$ is strict if it is continuous and strictly decreasing, i.e.,  $N(x)<N(y)$ when $y<x$. A fuzzy negation $N$ satisfying the condition N3 is called strong.

N3:  $N(N(x))=x$ for each $x\in [0,1]$.

A fuzzy negation is called crisp if it satisfies N4

N4:  For all $x\in[0,1], N(x)\in \{0,1\}$.\\

\begin{example} 
Some examples of fuzzy negations:

1. Standard negation: $N^S(x)=1-x$;

2. A strict non-strong negation: $N^{S^2}(x)=1-x^2$;

3. A strong non-standard negation:  $N_S^2(x)=\sqrt{1-x^2}$.
\hfill\rule{2mm}{2mm}\\
\end{example}

Note that:
\begin{enumerate}
\item If $N$ is strong then it has an inverse $N^{-1}$ which is also a strict fuzzy negation;
\item If $N$ is strong then $N$ is strict.
\end{enumerate}

A fuzzy negation $N$ is said to be \textit{non-vanishing} if $N(x)=0$ iff $x=1$ and $N$ is said to be \textit{non-filling} if $N(x)=1$ iff $x=0$. A fuzzy negation $N$ that is simultaneously non-vanishing and non-filling is called \textit{frontier} \cite{DB14b}.\\

An \textit{equilibrium point} of a fuzzy negation $N$ is a value $e\in [0,1]$ such that $N(e) = e$.\\

\begin{definition}\cite[Definition 2.3.1]{BJ08}
Let $T$ be a  t-norm. The function $N_T: [0,1]\rightarrow [0,1]$ defined as 
\begin{eqnarray*}
N_T(x)=\sup\{y \in [0,1]: T(x, y)=0\}
\end{eqnarray*}
 for each $x\in [0,1]$, is called the \textit{natural  fuzzy negation of $T$} or the negation induced by $T$. In addition, let $S$ be a  t-conorm. The function $N_S: [0,1]\rightarrow [0,1]$ defined as 
\begin{eqnarray*} 
 N_S(x)=\inf\{y \in [0,1]: S(x, y)=1\}
\end{eqnarray*} 
  for each $x\in [0,1]$, is called the \textit{natural  fuzzy negation of $S$} or the negation induced by $S$.\\
\end{definition}

\begin{remark}\cite[Remark 2.3.2 (i)]{BJ08}\label{remNT}
Clearly $N_T$ and $N_S$ are, in fact, fuzzy negations.\\
\end{remark}

\begin{definition}\cite[Definition 1.4.15]{BJ08}\label{DNJ}
Let $J$ be a fuzzy implication. The function  $N_J: [0,1]\rightarrow [0,1]$ defined by
\begin{eqnarray}\label{NJ}
N_J(x) = J(x, 0)
\end{eqnarray}
for all $x\in [0,1]$, is called the natural negation of $J$ or the negation induced by $J$.\\
\end{definition}

\begin{remark}\cite[Lemma 1.4.14]{BJ08}\label{remNJ}
Clearly $N_J$ is in fact a fuzzy implications.\\
\end{remark}





\begin{definition}\cite[Definition 2.5]{BSJIFS} 
Let $T$ be a  t-norm, $S$ be a t-conorm and $N$ be a strict fuzzy negation. $T_N$ is the $N$-dual of $T$ if, for all $x, y\in [0,1]$, $T_N(x,y)=N^{-1}(T(N(x), N(y)))$. Similarly, $S_N$ is the $N$-dual of $S$ if, for all $x, y\in [0,1]$, $S_N(x,y)=N^{-1}(S(N(x),$ $N(y)))$.\\
\end{definition}

\begin{proposition}\cite[Theorem 3.2]{weber}
Let $T$ be a  t-norm, $S$ be a t-conorm and $N$ be a strict fuzzy negation. Then, $T_N$ is a t-conorm and $S_N$ is a t-norm.\\
\end{proposition}

If the negation is standard, then $T_N$ is called dual t-conorm of $T$ and  $S_N$ is called dual t-norm of $S$.

\subsection{Ordinal sums of t-norms, t-conorms and fuzzy implications}

In this subsection, we will introduce the notion of ordinal sums of a family of t-norms and t-conorms, and some important results that will be used in the course of this work. For more information, see \cite{BSJIFS, klement, weber}.\\

\begin{proposition}\cite{klement}\label{PTSum}
Let $(T_i)_{i\in I}$ be a family of t-norms and $(]a_i, b_i[)_{i\in I}$ be a family of nonempty, pairwise disjoint open subintervals of $[0,1]$. Then the function $T_I:[0,1]^2 \rightarrow [0,1]$ defined by
\begin{eqnarray}\label{TSum}
T_I(x,y)= \left\{
\begin{array}{ll}
a_i+(b_i-a_i) ~T_i\left(\dfrac{x-a_i}{b_i-a_i}, \dfrac{y-a_i}{b_i-a_i}\right) & \mbox{if}~  (x, y)\in [a_i, b_i]^2,  \\
\min(x, y) & \mbox{otherwise}.
\end{array}
\right.
\end{eqnarray}
is a t-norm which is called the ordinal sum of the summands $(a_i, b_i, T_i)_{i\in I}$.\\
\end{proposition}

\begin{proposition}\cite{klement}\label{PSSum}
Let $(S_i)_{i\in I}$ be a family of t-conorms and $(]a_i, b_i[)_{i\in I}$ be a family of nonempty, pairwise disjoint open subintervals of $[0,1]$. 
Then the function $S_I:[0,1]^2 \rightarrow [0,1]$ defined by
\begin{eqnarray}\label{SSum}
S_I(x,y)=\left\{
\begin{array}{ll}
a_i+(b_i-a_i) ~S_i\left(\dfrac{x-a_i}{b_i-a_i}, \dfrac{y-a_i}{b_i-a_i}\right)  & \mbox{if}~  (x, y)\in [a_i, b_i]^2,  \\
\max(x, y) & \mbox{otherwise}.
\end{array}
\right.
\end{eqnarray}
is a t-conorm which is called the ordinal sum of the summands $(a_i, b_i, S_i)_{i\in I}$.\\
\end{proposition}

Nevertheless, for fuzzy implications there are several proposal of ordinal sums. For example,\\


%
\begin{proposition}\cite[Theorem 7]{DK161}\label{theo-JI-RS}
Let $(J_i)_{i\in I}$ be a family of implications and $(a_i, b_i)_{i\in I}$ be a family of nonempty  pairwise disjoint open 
subintervals of 
$[0, 1]$  such that $a_i>0$ for each $i\in I$. Then the function $J_I: [0, 1]^2 \rightarrow [0, 1]$ given by
\begin{eqnarray}
J_I(x,y)=\left\{
\begin{array}{ll}
a_i+(b_i-a_i) ~J_i\left(\dfrac{x-a_i}{b_i-a_i}, \dfrac{y-a_i}{b_i-a_i}\right) & \mbox{if}~  x, y\in [a_i, b_i], \\
J_{RS} (x, y) & \mbox{otherwise}.
\end{array}
\right.
\end{eqnarray}
is an implication which is called the ordinal sum of the summands $(a_i, b_i, J_i)_{i\in I}$.\\
\end{proposition}

Other proposal of ordinal sums for fuzzy implications can be found, for example, in \cite{B17,DK16,DK161,DK17,Su15}.

\section{Ordinal sums of fuzzy negations}

In this section, we will use the definition of ordinal sums of fuzzy negations  introduced in \cite{BSJIFS}, to show some results involving equilibrium point, relations between some classes of fuzzy negations and that ordinal sum of a fuzzy negation family is a fuzzy negation.\\

\begin{definition}\cite[Definition 3.1]{BSJIFS}\label{def-ODN}
 Let $(N_i)_{i\in I}$ be a family of fuzzy negations and $(]a_i, b_i[)_{i\in I}$ be a family of nonempty, pairwise disjoint open 
subintervals of $[0,1]$. Then the function $N_I:[0,1] \rightarrow [0,1]$ defined by
\begin{eqnarray}\label{NSum}
N_I(x)=\left\{
\begin{array}{ll}
(1-b_i)+(b_i-a_i) ~N_i\left(\dfrac{x-a_i}{b_i-a_i}\right) & \mbox{if}~  x\in [a_i, b_i],  \\
N^S(x) & \mbox{otherwise}.
\end{array}
\right.
\end{eqnarray}
is  called of the ordinal sum of the summands $(a_i, b_i, N_i)_{i\in I}$.\\
\end{definition}

\begin{lemma}\cite[Lemma 3.1]{BSJIFS}\label{LNSum}
Let $(]a_i, b_i[)_{i\in I}$ be a family of nonempty, pairwise disjoint open 
subintervals of $[0,1]$, $(N_i)_{i\in I}$ be a family of fuzzy negations and $N_I$ the ordinal sum $N_I$ of the summands $(a_i, b_i, N_i)_{i\in I}$.
Then,

1) If $x\in [a_i,b_i]$ for some $i\in I$, then $N_I(x)\in [1-b_i,1-a_i]$;

2) If $x\not\in \bigcup_{i\in I} [a_i,b_i]$, then $N_I(x)\not\in \bigcup_{i\in I} [1-b_i,1-a_i]$.\\
\end{lemma}

\begin{proposition}\cite[Proposition 3.1]{BSJIFS}\label{PNSum-old}
Let $(]a_i, b_i[)_{i\in I}$ be a family of nonempty, pairwise disjoint open subintervals of $[0,1]$ and $(N_i)_{i\in I}$ be a family of fuzzy negations.
Then the ordinal sum $N_I$ of the summands $(a_i, b_i, N_i)_{i\in I}$ is a fuzzy negation.\\
\end{proposition}

If $(N_i)_{i\in I}$ is a family of fuzzy negations such that $(]N_i(b_i),N_i( a_i)[)_{i\in I}$ is also a family of nonempty, 
pairwise disjoint open subintervals of $[0,1]$, then the ordinal sum of $(T_i)_{i\in I}$ and  $(S_i)_{i\in I}$ with respect to 
$(]N_i(b_i),N_i( a_i)[)_{i\in I}$ will be denoted by $T_I^N$ and $S_I^N$, respectively.\\

\begin{proposition}
Let $(]a_i, b_i[)_{i\in I}$ be a family of nonempty, pairwise disjoint open subintervals of $[0,1]$, $(N_i)_{i\in I}$ be a family of fuzzy negations 
and $N_I$ be the ordinal sum of the summands $(a_i, b_i, N_i)_{i\in I}$.  If, for some $i\in I$, $N_i$ has an equilibrium point $e_i$ and $b_i=1-a_i$,
then $a_i+(b_i-a_i)e_i$ is the equibibrium point of $N_I$.
\end{proposition}

\begin{proof}
Suppose that $N_i(e_i)=e_i$, for all $i=1, \ldots, n$. Then, $x_i=a_i+(b_i-a_i)e_i\in [a_i,b_i]$ and therefore
 $N_I(x_i) = (1-b_i)+(b_i-a_i) ~N_i\left(\dfrac{x_i-a_i}{b_i-a_i}\right)$. Since, $\dfrac{x_i-a_i}{b_i-a_i}=e_i$ and $a_i=1-b_i$, then 
\begin{eqnarray*}
N_I\left(a_i+(b_i-a_i)e_i\right) & = & 1-b_i+(b_i-a_i) ~N_i\left(e_i\right)\\
                                                & = & 1-b_i+(b_i-a_i) e_i  \\   
                                                   & = & a_i+(b_i-a_i) e_i 
\end{eqnarray*}
Therefore,  $a_i+(b_i-a_i)e_i$ is the equilibrium point of $N_I$.
\end{proof}\\

\begin{proposition}\label{Ni<NS}
Let $(]a_i, b_i[)_{i\in I}$ be a family of nonempty, pairwise disjoint open subintervals of $[0,1]$, $(N_i)_{i\in I}$ be a family of fuzzy negations
and $N_I$ be the ordinal sum of the summands $(a_i, b_i, N_i)_{i\in I}$.  Then  $N_I\leq N^S$ if and only if  $N_i\leq N^S$ for all $i\in I$.
\end{proposition}

\begin{proof}
($\Rightarrow$) Let $i\in I$ and $x\in [0,1]$. Then $x_i=a_i+(b_i-a_i)x\in [a_i,b_i]$. So,

\begin{eqnarray*}
 N_i(x) & = & N_i\left(\frac{x_i-a_i}{b_i-a_i}\right) \\
 & = & \frac{N_I(x_i)-(1-b_i)}{b_i-a_i} \\
 & \leq & \frac{N^S(x_i)-(1-b_i)}{b_i-a_i} \\
 & = & \frac{b_i-x_i}{b_i-a_i} \\
  & = & 1-x = N^S(x)
\end{eqnarray*}

($\Leftarrow$)
Suppose that $N_i\leq N^S$ for all $i\in I$. Then
\begin{eqnarray*}
N_I(x) & = & \left\{
\begin{array}{ll}
(1-b_i)+(b_i-a_i) ~N_i\left(\dfrac{x-a_i}{b_i-a_i}\right) & \mbox{if}~  x\in [a_i, b_i],  \\
N^S(x) & \mbox{otherwise}.
\end{array}
\right.\\
& \leq & \left\{
\begin{array}{ll}
(1-b_i)+(b_i-a_i) ~N^S\left(\dfrac{x-a_i}{b_i-a_i}\right) & \mbox{if}~  x\in [a_i, b_i],  \\
N^S(x) & \mbox{otherwise}.
\end{array}
\right.\\
& = & N^S(x).
\end{eqnarray*} 
\end{proof}\\

\begin{proposition}
Let $(]a_i, b_i[)_{i\in I}$ be a family of nonempty, pairwise disjoint open subintervals of $[0,1]$, $(N_i)_{i\in I}$ be a family of fuzzy negations 
and $N_I$ be the ordinal sum $N_I$ be the summands $(a_i, b_i, N_i)_{i\in I}$. Then,   $N_I\geq N^S$ if and only if  $N_i\geq N^S$ for all $i\in I$.
\end{proposition}

\begin{proof}
Analogous from Proposition \ref{Ni<NS}.
\end{proof}\\

\section{Ordinal sums of fuzzy negations and classes of fuzzy negations}

In this section, we will prove some propositions and theorems using definitions and results introduced in the previous sections. We will establish conditions for the ordinal sum of a family of fuzzy negations resulting in a fuzzy negation belonging to a class of fuzzy negations such as strict, strong, frontier, continuous and invertible. \\

\begin{proposition}\label{PNSum-new}
Let $(]a_i, b_i[)_{i\in I}$ be a family of nonempty, pairwise disjoint open 
subintervals of $[0,1]$ and $(N_i)_{i\in I}$ be a family of functions and $N_I$ the function obtained as in Eq. (\ref{NSum}).
All the $N_i$'s are (continuous, strictly decreasing) fuzzy negations if and only if then $N_I$ is a (continuous, strictly decreasing) 
fuzzy negation such that $N_I(a_i)=1-a_i$ and  $N_I(b_i)=1-b_i$ for each $i\in I$.
\end{proposition}

\begin{proof} ($\Rightarrow$) If all the $N_i$'s are fuzzy negations, then, by Proposition \ref{PNSum-old}, $N_I$ is a fuzzy negation. In addition, for each $i\in I$, 
$N_I(a_i)=1-b_i+(b_i-a_i)N_i\left(\frac{a_i-a_i}{b_i-a_i}\right)=1-b_i+(b_i-a_i)N_i(0)=1-a_i$.

Now, suppose that for each $i\in I$,  $N_i$ is continuous. Then, $N_{I|_{[a_i, b_i]}}(x) = (1{-}b_i)+(b_i{-}a_i) ~N_i\left(\dfrac{x-a_i}{b_i-a_i}\right)$ is clearly continuous. Since, $N^S$ is continuous then it is sufficient to prove that for each $i\in I$,
$\lim\limits_{x\searrow a_i} N^S( x)=N_I(a_i)$ and $\lim\limits_{x\nearrow b_i} N^S(x)=N_I(b_i)$. 
In fact, 
\begin{eqnarray*}
\lim\limits_{x\searrow a_i}N^S(x)
  & =&\lim\limits_{x\searrow a_i}1-x \\
  & =&1-a_i \\
  & = &(1-b_i)+(b_i-a_i) ~N_i\left(\dfrac{a_i-a_i}{b_i-a_i}\right) \\
  & = & N_I(a_i).
 \end{eqnarray*}
Analogously, we prove that $N^S\left (\lim\limits_{x\nearrow b_i} x\right )=N_I(b_i)$. 

Now we will prove that $N_I$ is strictly decreasing when all $N_i$ are strictly decreasing. If $x< y$ then we have the following cases:
\begin{description}
 \item[Case 1:]~ If $x,y\in [a_i,b_i]$ for some $i\in I$, then $\dfrac{x-a_i}{b_i-a_i}< \dfrac{y-a_i}{b_i-a_i}$ and therefore 
 $N_i\left(\dfrac{y-a_i}{b_i-a_i}\right) < N_i\left(\dfrac{x-a_i}{b_i-a_i}\right)$. So, by  Eq. (\ref{NSum}), 
 $N_I(y)< N_I(x)$.
 \item[Case 2:]~ If $x\in [a_i,b_i]$ and $y\in [a_j,b_j]$ for some $i,j\in I$ such that $i\neq j$ then $a_i< b_i\leq a_j< b_j$. So, by Lemma \ref{LNSum}, 
 $N_I(y)\in [1-b_j,1-a_j]$ and $N_I(x)\in [1-b_i,1-a_i]$. Thus, since $1-b_j< 1-a_i$, then $N_I(y)< N_I(x)$.
 \item[Case 3:]~ If $x\in [a_i,b_i]$ for some $i\in I$ and $y\not \in \bigcup\limits_{j\in I} [a_j,b_j]$, then $b_i<y$ and therefore $1-y< 1-b_i$. 
 Since, by Lemma \ref{LNSum}, $1-b_i\leq N_I(x)$ and by Eq. (\ref{NSum}) $N_I(y)=1-y$, then follows that $N_I(y)< N_I(x)$.
 \item[Case 4:]~ If $x\not\in \bigcup_{j\in I} [a_j,b_j]$ and $y\in [a_i,b_i]$ for some $i\in I$, then $x< a_i$ and,
 by Eq. (\ref{NSum}) $N_I(x)=1-x$. Therefore,  by Lemma  \ref{LNSum},  $N_I(y)\leq 1-a_i< 1-x=N_I(x)$.
 \item[Case 5:]~ If $x,y\not \in \bigcup\limits_{i\in I} [a_i,b_i]$ then by  Eq. (\ref{NSum}), $N_I(y)=1-y< 1-x=N_I(x)$.
\end{description}
Therefore, $N_I$ is strictly decreasing.

($\Leftarrow$) If $N_I$ is a fuzzy negation such that $N_I(a_i)=1-a_i$ and $N_I(b_i)=1-b_i$ then, for each  $i\in I$, 
\begin{eqnarray*}
N_i(0) & = & N_i\left(\frac{a_i-a_i}{b_i-a_i}\right)\\
& = & \frac{-(1-b_i)+(1-b_i)+(b_i-a_i)N_i\left(\frac{a_i-a_i}{b_i-a_i}\right)}{b_i-a_i}\\
& = & \frac{N_I(a_i)-(1-b_i)}{b_i-a_i}\\
& = & \frac{1-a_i-(1-b_i)}{b_i-a_i}\\
& = & 1.
\end{eqnarray*} 
 Analogously, 
\begin{eqnarray*} 
N_i(1) & = & N_i\left(\frac{b_i-a_i}{b_i-a_i}\right)\\
& = & \frac{-(1-b_i)+(1-b_i)+(b_i-a_i)N_i\left(\frac{b_i-a_i}{b_i-a_i}\right)}{b_i-a_i}\\
& = & \frac{N_I(b_i)-(1-b_i)}{b_i-a_i}\\
& = &  \frac{1-b_i-(1-b_i)}{b_i-a_i}\\
& = & 0.
\end{eqnarray*} 
Let $i\in I$, $x,y\in [0,1]$ such that $x\leq y$ and $x_i=a_i+(b_i-a_i)x$ and $y_i=a_i+(b_i-a_i)y$.  Then, $a_i\leq x_i\leq y_i\leq b_i$ and therefore, 
$N_I(y_i)\leq N_I(x_i)$. So, $1-b_i+(b_i-a_i)N_i(y)=1-b_i+(b_i-a_i)N_i\left(\frac{y_i-a_i}{b_i-a_i}\right)\leq 
1-b_i+(b_i-a_i)N_i\left(\frac{x_i-a_i}{b_i-a_i}\right)=1-b_i+(b_i-a_i)N_i(x)$. Hence, $N_i(y)\leq N_i(x)$. 
Therefore, $N_i$ is a fuzzy negation for each $i\in I$.
In addition, from Eq.  (\ref{NSum}), clearly for each $i\in I$, if $N_I$ is continuous then $N_i$ also is continuous and if 
$N_I$ is strictly then $N_i$ also is strict. 
\end{proof}\\

\begin{proposition}\label{PNSum-new2}
Let $(]a_i, b_i[)_{i\in I}$ be a family of nonempty, pairwise disjoint open 
subintervals of $[0,1]$ and $(N_i)_{i\in I}$ be a family of  fuzzy negations such that $N_i$ is non-filling when $a_i=0$ and non-vanishing when $b_i=1$. 
Then, the ordinal sum $N_I$ of the summands 
$(a_i, b_i, N_i)_{i\in I}$ is frontier.
\end{proposition}

\begin{proof}
 By Proposition \ref{PNSum-old}, $N_I$ is a fuzzy negation. Let $x\in ]0,1[$. If $x\not\in \bigcup\limits_{i\in I} [a_i,b_i]$ 
 then, from Eq.  (\ref{NSum}), $N_I(x)=1-x\in ]0,1[$. Suppose that $x\in [a_i,b_i]$ for some $i\in I$. If  $a_i\neq 0$ and $b_i\ne1 1$ 
 then, by  Lemma \ref{LNSum}, 
 $1 > b_i \geq N_I(x)\geq a_i >0$, i.e.  $N_I(x)\in ]0,1[$.  If  $a_i=0$ then $\frac{x-a_i}{b_i-a_i}=\frac{x}{b_i}\in ]0,1[$. So, 
 because $N_i$ is non-filling, $N_I(x)=1-b_i+b_iN_i(\frac{x}{b_i}) < 1$. Analogously, if $b_i=1$ then $\frac{x-a_i}{b_i-a_i}=\frac{x-a_i}{1-a_i}\in ]0,1[$.
 So,  because $N_i$ is non-vanishing, $N_I(x)=(1-a_i)N_i(\frac{x-a_i}{1-a_i}) > 1$. Therefore, for each $x\in ]0,1[$, $N_I(x)\in ]0,1[$, i.e. 
 $N_I$ is frontier.
\end{proof}\\

\begin{proposition}\label{PNSum-new3}
Let $(]a_i, b_i[)_{i\in I}$ be a family of nonempty, pairwise disjoint open subintervals of $[0,1]$ and $(N_i)_{i\in I}$ be a family of fuzzy negations. 
If the ordinal sum $N_I$ of the summands $(a_i, b_i, N_i)_{i\in I}$ satisfies the following two properties 
\begin{enumerate}
 \item $N_I(x)=1-a_i$ for some $i\in I$ only when $x= a_i$; and
 \item $N_I(x)=1-b_i$ for some $i\in I$ only when $x= b_i$
\end{enumerate}
then $N_i$ is frontier for each $i\in I$.
\end{proposition}

\begin{proof}
 Suppose that  for some $i\in I$, $N_i$ is not frontier. Then there exists $x\neq 1$ such that $N_i(x)=0$ or 
there exists $x\neq 0$ such that $N_i(x)=1$.  Let $x_i=a_i+(b_i-a_i)x$ and therefore, $x_i\in [a_i,b_i]$. In the first case, we have that 
$N_I(x_i)=1-b_i+(b_i-a_i)N_i\left(\frac{x_i-a_i}{b_i-a_i}\right) =
1-b_i+(b_i-a_i)N_i(x)=1-b_i$. So, by the second property, $x_i=b_i$ and therefore, $x=1$ which is a contradiction.
The second case is analogous.
Therefore, for each $i\in I$, the fuzzy negation  $N_i$ is frontier.
\end{proof}\\

\begin{theorem}\label{PNSum-strong-negationb}
Let $(]a_i, b_i[)_{i\in I}$ be a family of nonempty, pairwise disjoint open subintervals of $[0,1]$ and $(N_i)_{i\in I}$ be a family of fuzzy negations.
If all the $N_i$'s are strict fuzzy negations and, for each $i\in I$, there exists $j\in I$ such that $[a_j,b_j]=[1-b_i,1-a_i]$ and $N_j=N_i^{-1}$, then the ordinal sum $N_I$ of the summands $(a_i, b_i, N_i)_{i\in I}$, is a strong fuzzy negation.
\end{theorem}

\begin{proof}
From Proposition \ref{PNSum-old}, $N_I$ is a fuzzy negation. Besides this, for any $x\in [0,1]$ if $x\in [a_i,b_i]$ for some $i\in I$ then 
 for hypothesis there exists $j\in I$ such that $[a_j,b_j]=[1-b_i,1-a_i]$ and by Eq. (\ref{NSum}),
$N_I(x)=(1-b_i)+(b_i-a_i) ~N_i\left(\dfrac{x-a_i}{b_i-a_i}\right) \in [1-b_i,1-a_i]=[a_j,b_j]$. Therefore,
\begin{eqnarray*}
  N_I(N_I(x)) & = & \hspace{-0.25cm} (1{-}b_j){+}(b_j{-}a_j)N_j\left( 
   \dfrac{(1{-}b_i){+}(b_i{-}a_i)N_i\left(\dfrac{x{-}a_i}{b_i{-}a_i}\right){-}a_j}{b_j-a_j}   
   \right ) \\
& = &  a_i{+}(b_i{-}a_i) N_i^{-1}\left( 
   \dfrac{(1{-}b_i){+}(b_i{-}a_i)N_i\left(\dfrac{x{-}a_i}{b_i{-}a_i}\right){-}(1{-}b_i)}{b_i-a_i}   
   \right ) \\ 
   & = & x.
  \end{eqnarray*}
If $x\not\in\bigcup\limits_{i\in I} [a_i,b_i]$ then $N_I(x)=1-x$ and, by Lemma \ref{LNSum} and hypothesis,  
$N_I(x)\not\in\bigcup\limits_{i\in I} [1-b_i,1-a_i]=\bigcup\limits_{i\in I} [a_i,b_i]$. 
So, $N_I(N_I(x))=1-(1-x)=x$.
\end{proof}\\

\begin{theorem}\label{PNSum-strong-negation}
Let $(]a_i, b_i[)_{i\in I}$ be a family of nonempty, pairwise disjoint open 
subintervals of $[0,1]$ and $(N_i)_{i\in I}$ be a family of fuzzy negations.
If the ordinal sum $N_I$ of the summands $(a_i, b_i, N_i)_{i\in I}$,
is a strong fuzzy negation then  all the $N_i$'s are strict fuzzy negations and, for each $i\in I$, there exists $j\in I$ such that 
$[a_j,b_j]=[1-b_i,1-a_i]$ and $N_j=N_i^{-1}$. In addition, if for each $i,j\in I$, $a_i\neq b_j$  and $N_i\neq N^S$ then  
for each $i\in I$, there exists $j\in I$ such that $[a_j,b_j]=[1-b_i,1-a_i]$ and $N_j=N_i^{-1}$.
\end{theorem}

\begin{proof}
If $x<y$ and $i\in I$, then taking $x_i=a_i+(b_i-a_i)x$ and $y_i=a_i+(b_i-a_i)y$, we have that $x_i< y_i$, $x_i,y_i\in [a_i,b_i]$,
$N_i(x)=N_i\left(\dfrac{x_i-a_i}{b_i-a_i}\right)$ and $N_i(y)=N_i\left(\dfrac{y_i-a_i}{b_i-a_i}\right)$. Since, 
$N_I(y_i)<N_I(x_i)$, then by Eq. (\ref{NSum}), $(1-b_i)+(b_i-a_i)~N_i\left(\dfrac{y_i-a_i}{b_i-a_i}\right)<(1-b_i)+(b_i-a_i)~N_i\left(\dfrac{x_i-a_i}{b_i-a_i}\right)$. So, 
$N_i(y)=N_i\left(\dfrac{y_i-a_i}{b_i-a_i}\right)<N_i\left(\dfrac{x_i-a_i}{b_i-a_i}\right)=N_i(x)$ and therefore, each $N_i$ is strictly decreasing.
On the other hand, by Lemma \ref{LNSum} is clear that case some $N_i$ is non continuous then by Eq. (\ref{NSum}), $N_I$ also would be non continuous. 
Hence, each $N_i$ is strict. 

Now, suppose that for each $i,j\in I$, $a_i\neq b_j$ when $i\neq j$ and $N_i\neq N^S$. If $1-a_i\in \bigcup\limits_{j\in I} [a_j,b_j]$ 
 then  $1-a_i\in [a_j,b_j]$ for some $j\in I$. 
If $1-a_i < b_j$ then there exist $\epsilon> 0$ such that 
 $1-a_i+\epsilon< b_j$ but $a_i-\epsilon\not\in \bigcup\limits_{j\in I} [a_j,b_j]$.
 So, $N_I(a_i-\epsilon))= 1-a_i+\epsilon\in ]a_j,b_j[$ and therefore, 
 
 \begin{eqnarray*}
   a_i-\epsilon &=& N_I(N_I(a_i-\epsilon)) \\
   &=& 1-b_j+(b_j-a_j)N_j\left(\frac{(1-a_i+\epsilon)-a_j}{b_j-a_j}\right) \mbox{ by Eq. (\ref{NSum})} \\
   &\neq & 1-b_j+(b_j-a_j)\left(1- \frac{(1-a_i+\epsilon)-a_j}{b_j-a_j}\right)  \mbox{ since $N_J\neq N^S$}\\
   &=& a_i-\epsilon.    
  \end{eqnarray*}
  which is a contradiction, and therefore $1-a_i=b_j$. Analogously is possible to prove that $1-b_i=a_j$. 
  
  Suppose that $N_j\neq N_i^{-1}$ then there exists $x\in [0,1]$ such that $N_j(x)\neq N_i^{-1}(x)$ and therefore. Let $x_j=a_j+(b_j-a_j)x$ and therefore, 
  $x=\frac{x_j-a_j}{b_j-a_j}$.
  \begin{eqnarray*}
   N_I(N_I(x_j)) & = & N_I\left(1-b_j+(b_j-a_j)N_j\left(\frac{x_j-a_j}{b_j-a_j}\right)\right)\\
   & = & N_I(a_i+(b_i-a_i)N_j(x)) \\
   & \neq & N_I(a_i+(b_i-a_i)N_i^{-1}(x)) \\
   & = & a-b_i+(b_i-a_i)N_i\left(\frac{a_i+(b_i-a_i)N_i^{-1}(x)-a_i}{b_i-a_i}\right) \\
   & = & a-b_i+(b_i-a_i)x \\
   & = & x_j
  \end{eqnarray*}
which is a contradiction once $N_I$ is strong. Therefore,  $N_j=N_i^{-1}$.
\end{proof}\\

\begin{corollary}
 Let $(]a_i, b_i[)_{i\in I}$ be a family of nonempty, pairwise disjoint open 
subintervals of $[0,1]$ and $(N_i)_{i\in I}$ be a family of fuzzy negations such that for each $i,j\in I$, $a_i\neq b_j$  and $N_i\neq N^S$.
Then, the ordinal sum $N_I$ of the summands $(a_i, b_i, N_i)_{i\in I}$,
is a strong fuzzy negation if and only if   all the $N_i$'s are strict fuzzy negations and, for each $i\in I$, there exists $j\in I$ such that 
$[a_j,b_j]=[1-b_i,1-a_i]$ and $N_j=N_i^{-1}$. 
\end{corollary}

\begin{proof}
 Straighforward from Theorems \ref{PNSum-strong-negationb} and \ref{PNSum-strong-negation}. 
\end{proof}\\

\begin{proposition}
Let $(]a_i, b_i[)_{i\in I}$ be a family of nonempty, pairwise disjoint open subintervals of $[0,1]$, $(N_i)_{i\in I}$ be a family of inversible fuzzy negations, $N_I$ be the ordinal sums of the summands $(a_i, b_i, N_i)_{i\in I}$ and $N^{-1}_{I}$ be the ordinal sums of the summands 
$(1-b_i, 1-a_i, N_i^{-1})_{i\in I}$. 
Then, $N^{-1}_{I}$ is the inverse of $N_I$.
\end{proposition}

\begin{proof}
It is sufficient to prove that $N_I \circ N^{-1}_{I}(x) = N^{-1}_{I}\circ N_I(x) = x$. Then
\begin{eqnarray*}
N_I \circ N^{-1}_{I}(x)& = &  N_I\left(\left\{\hspace{-0.1cm}
\begin{array}{ll}
1{-}(1{-}a_i){+}(1{-}a_i{-}(1{-}b_i))N_i^{-1}\left(\dfrac{x{-}(1{-}b_i)}{1{-}a_i{-}(1{-}b_i)}\right) & \mbox{if}~  x\in [1-b_i, 1-a_i],  \\
N^S(x) & \mbox{otherwise}.
\end{array}
\right.\hspace{-0.2cm}\right)\\
& = & N_I\left(\left\{\hspace{-0.1cm}
\begin{array}{ll}
a_i{+}(b_i{-}a_i)N_i^{-1}\left(\dfrac{x{-}1{+}b_i}{b_i{-}a_i}\right) &  \mbox{if}~  x\in [1{-}b_i, 1{-}a_i] \\
N^S(x) & \mbox{otherwise}.
\end{array}
\right.\right)\\
\end{eqnarray*}
Let $y= N^{-1}_I(x)$, then $y\in [a_i,b_i]$ if and only if $x\in [1-b_i,1-a_i]$.  Therefore,

\begin{eqnarray*}
N_I \circ N^{-1}_{I}(x) =  N_I(y) & \hspace{-0.3cm}= \hspace{-0.3cm} & \hspace{-0.3cm}
\left\{\hspace{-0.2cm}
\begin{array}{ll}
(1{-}b_i){+}(b_i{-}a_i) N_i\left(\dfrac{a_i{+}(b_i{-}a_i) N_i^{-1}\left(\dfrac{x{-}1{+}b_i}{b_i{-}a_i}\right){-}a_i}{b_i{-}a_i}\right) & \mbox{if}~  y\in [a_i, b_i],  \\
N^S(N^S(x)) & \mbox{otherwise}.
\end{array}
\right.\\
\hspace{-0.3cm}& \hspace{-0.3cm}= \hspace{-0.3cm} & \hspace{-0.3cm} \left\{
\begin{array}{ll}
(1-b_i)+(b_i-a_i) ~N_i\left(N_i^{-1}\left(\dfrac{x-1+b_i}{b_i-a_i}\right)\right) & \mbox{if}~  x\in [a_i, b_i],  \\
x & \mbox{otherwise}.
\end{array}
\right.\\
& = & x.
\end{eqnarray*}

Analogously, we prove that $N^{-1}_{I}\circ N_I(x) = x$. 
\end{proof}\\

Observe that, if $(]a_i, b_i[)_{i\in I}$ is a family of nonempty, pairwise disjoint open subintervals of $[0,1]$ and $(N_i)_{i\in I}$ is a family of fuzzy negations, then there exists a family $(]c_j, d_j[)_{i\in J}$ of nonempty, pairwise disjoint open subintervals of $[0,1]$ and a family $(N'_j)_{j\in J}$ of fuzzy negations where for each $l,j\in J$, $c_l\neq d_j$  and $N'_j\neq N^S$, such that $N_I=N'_J$.\\

\section{Ordinal sums of fuzzy negations and the ordinal sums of t-norms, t-conorms and fuzzy implications}

\begin{proposition}
Let $(T_i)_{i\in I}$ be a family of t-norms and $(]a_i, b_i[)_{i\in I}$ be a family of nonempty, pairwise disjoint open subintervals of $[0,1]$. Then 
\begin{eqnarray}\label{NSum}
N_{T_I}(x)=\left\{
\begin{array}{ll}
0 & \mbox{if}~  x\in [a_i, b_i]\mbox{ and } a_i> 0 ~\mbox{or}  x\not\in \bigcup_{i\in I} [a_i,b_i] \cup \{0\} \\~
1 & \mbox{if}~x=0 \\
N_{T_i}\left(\dfrac{x}{b_i}\right)  & \mbox{if}~  x\in ]a_i, b_i]~\mbox{and} ~a_i= 0.
\end{array}
\right. 
\end{eqnarray}
\end{proposition}

\begin{proof}
If $x\in [a_i, b_i]$ and $a_i> 0$, then
\begin{eqnarray*}
 N_{T_I}(x) & \hspace{-0.3cm}= \hspace{-0.3cm} &  \sup\left\{y\in [0,1]:  T_I\left(x, y\right)=0\right\}\\
\hspace{-0.3cm}& \hspace{-0.3cm}= \hspace{-0.3cm} & \sup\left\{y{\in} [0,1]: a_i {+} (b_i{-}a_i) T_i\left(\dfrac{x-a_i}{b_i-a_i}, \dfrac{y-a_i}{b_i-a_i}\right){=}0\right\}\\
\hspace{-0.3cm}& \hspace{-0.3cm}= \hspace{-0.3cm} & \sup\emptyset \\  
\hspace{-0.3cm}& \hspace{-0.3cm}= \hspace{-0.3cm} &0. 
\end{eqnarray*}

If $x\not\in \bigcup_{i\in I} [a_i,b_i] \cup \{0\}$, then
\begin{eqnarray*}
N_{T_I}(x) & = & \sup\left\{y\in [0,1]:  T_I\left(x, y\right)=0\right\}\\
                 & = & \sup\left\{y\in [0,1]:  \min(x, y)=0\right\}\\
                 & = & 0.
\end{eqnarray*}

If $x= 0$ then, trivially by Remark \ref{remNT} and Proposition \ref{PTSum}, $N_{T_I}(x)=1$. 

If $x\in ]a_i, b_i]$ and $a_i= 0$, then
\begin{eqnarray*}
N_{T_I}(x) 
& \hspace{-0.3cm}= \hspace{-0.3cm} &\sup\left\{y\in [0,1]:  T_I\left(x, y\right)=0\right\}\\
& \hspace{-0.3cm}= \hspace{-0.3cm}  & \sup\left\{y\in [0,b_i]: b_i T_i\left(\dfrac{x}{b_i},\dfrac{y}{b_i}\right)= 0\right\} \cup \{y\in ]b_i,1]:\min(x,y)\} \\ 
&  \hspace{-0.3cm}= \hspace{-0.3cm}  & \sup\left\{\dfrac{y}{b_i}\in [0,1]: T_i\left(\dfrac{x}{b_i}, \dfrac{y}{b_i}\right)=0\right\}\\
&  \hspace{-0.3cm}= \hspace{-0.3cm}  & \sup\left\{z\in [0,1]: T_i\left(\dfrac{x}{b_i}, z\right)=0\right\}\\
&  \hspace{-0.3cm}= \hspace{-0.3cm}  & N_{T_i}\left(\dfrac{x}{b_i}\right). 
\end{eqnarray*}
\end{proof}


\begin{proposition}
Let $(S_i)_{i\in I}$ be a family of t-conorms and $(]a_i, b_i[)_{i\in I}$ be a family of nonempty, pairwise disjoint open subintervals of $[0,1]$. Then 
\begin{eqnarray}\label{NSum}
N_{S_I}(x)=\left\{
\begin{array}{ll}
1 & \mbox{if}~  x\in [a_i, b_i], b_i<1 ~\mbox{or} x\not\in \bigcup_{i\in I} [a_i,b_i] \cup \{1\} \\
0 & \mbox{if}~x=1 \\
(1-a_i) N_{S_i}\left(\dfrac{x-a_i}{1-a_i}\right)+a_i & \mbox{if}~ x\in [a_i, b_i[~\mbox{and} ~b_i=1.
\end{array}
\right.
\end{eqnarray}
\end{proposition}

\begin{proof}
If $x\in [a_i, b_i]$ and $b_i<1$, then
\begin{eqnarray*}
 N_{S_I}(x)
& \hspace{-0.3cm}= \hspace{-0.3cm}  & \inf\left\{y\in [0,1]:  S_I\left(x, y\right)=1\right\}\\
& \hspace{-0.3cm}= \hspace{-0.3cm}  & \inf\left\{y\in [a_i, b_i]: a_i{+}(b_i{-}a_i) S_i\left(\dfrac{x{-}a_i}{b_i{-}a_i},\dfrac{y{-}a_i}{b_i{-}a_i}\right)= 1\right\}  \cup ~\{y\in [0,a_i[\cup]b_i, 1]:\max(x, y)=1\}\\
& \hspace{-0.3cm}= \hspace{-0.3cm}   & \inf\emptyset\cup \{1\} \\  
& \hspace{-0.3cm}= \hspace{-0.3cm}   & 1. 
\end{eqnarray*}

If $x\not\in \bigcup_{i\in I} [a_i,b_i] \cup \{1\}$, then
\begin{eqnarray*}
N_{S_I}(x) & = & \inf\left\{y\in [0,1]:  S_I\left(x, y\right)=1\right\}\\
                 & = & \inf\left\{y\in [0,1]:  \max(x, y)=1 \right\}\\
                 & = & 1.
\end{eqnarray*}

If $x= 1$ then, trivially by Remark \ref{remNT} and Proposition \ref{PSSum}, $N_{S_I}(x)=0$. 

If $x\in [a_i, b_i[$ and $b_i= 1$, then
\begin{eqnarray*}
 N_{S_I}(x) 
& \hspace{-0.3cm}= \hspace{-0.3cm}   & \inf\left\{y\in [0,1]:  S_I\left(x, y\right)=1\right\} \\
& \hspace{-0.3cm}= \hspace{-0.3cm}  & \inf\left\{y\in [a_i, 1]: a_i+(1-a_i) S_i\left(\dfrac{x{-}a_i}{1{-}a_i}{,}\dfrac{y{-}a_i}{1{-}a_i}\right)=1\right\} \cup~ \{y\in [0,a_i[: \max(x, y)=1\}\\
& \hspace{-0.3cm}= \hspace{-0.3cm}   & (1{-}a_i) \inf\left\{\dfrac{y-a_i}{1-a_i}\in [0, 1]:  S_i\left(\dfrac{x-a_i}{1-a_i}, \dfrac{y-a_i}{1-a_i}\right){=}1\right\} +a_i\\
& \hspace{-0.3cm}= \hspace{-0.3cm}   & (1-a_i) \inf\left\{z\in [0, 1]:  S_i\left(\dfrac{x-a_i}{1-a_i}, z\right)=1\right\}+a_i\\
& \hspace{-0.3cm}= \hspace{-0.3cm}   & (1-a_i) N_{S_i}\left(\dfrac{x-a_i}{1-a_i}\right)+a_i.
\end{eqnarray*}
\end{proof}

In the case of fuzzy implications, we have that the natural negation of the ordinal sums of a family of fuzzy implications, according to 
definition in the Theorem \ref{theo-JI-RS} (and also for the proposal in \cite{B17,DK16,DK161,DK17,Su15}) always result in $N_\bot$. So, in order to obtain,  
more interesting natural negations we propose the following notion of ordinal sums for fuzzy implications:\\

\begin{definition}\label{def-OSINova}
 Let $(J_i)_{i\in I}$ be a family of fuzzy implications and $(]a_i, b_i[)_{i\in I}$ be a family of nonempty, pairwise disjoint open 
 subintervals of $[0,1]$ such that $b_i< 1$ for each $i\in I$.  Then the function $J_I:[0,1]^2 \rightarrow [0,1]$ defined by
\begin{eqnarray}\label{NSum}
J_I(x,y)=\left\{
\begin{array}{ll}
(1-b_i)+(b_i-a_i) ~J_i\left(\dfrac{x-a_i}{b_i-a_i},y\right) & \mbox{if}~  x\in [a_i, b_i],  \\
J_{KD}(x,y) & \mbox{otherwise}.
\end{array}
\right.
\end{eqnarray}
is  called of the left ordinal sum of the summands $(a_i, b_i, J_i)_{i\in I}$.\\
\end{definition}

\begin{proposition}\label{pro-JI-nova}
        Let $(J_i)_{i\in I}$ be a family of fuzzy implications and $(]a_i, b_i[)_{i\in I}$ be a family of nonempty, pairwise disjoint open 
subintervals of $[0,1]$ such that $b_i< 1$ for each $i\in I$. Then the function $J_I:[0,1]^2 \rightarrow [0,1]$ defined in Equation (\ref{NSum}) is a 
fuzzy implication.
      \end{proposition}

      \begin{proof}
 J1) Let $x,y,z\in [0,1]$ such that $x\leq y$. We have several cases:
       \begin{enumerate}
        \item If $x, y \in [a_i,b_i]$ for some $i\in I$ then, because $\dfrac{x-a_i}{b_i-a_i}\leq \dfrac{y-a_i}{b_i-a_i}$ and $J_i$ satisfy J1, 
        we have that $J_i\left(\dfrac{y-a_i}{b_i-a_i},z\right)\leq J_i\left(\dfrac{x-a_i}{b_i-a_i},z\right)$ and therefore        
        $J_I(y,z)=(1-b_i)+(b_i-a_i) ~J_i\left(\dfrac{y-a_i}{b_i-a_i},z\right)\leq (1-b_i)+(b_i-a_i) ~J_i\left(\dfrac{x-a_i}{b_i-a_i},z\right)=J_I(x,z)$. 
      \item If $x\in [a_i,b_i]$ and $y\in [a_j,b_j]$ for some $i\neq j\in I$ then, $J_I(x,z)\in [1-b_i,1-a_i]$ and $J_I(y,z)\in [1-b_j,1-a_j]$. But, because 
      $x\leq y$ and $[a_i,b_i]$ be disjoint of $[a_j,b_j]$, 
      we have that $b_i< a_j$ and so $1-a_j< 1-b_i$. Therefore,  $J_I(y,z)< J_I(x,z)$.
  \item If $x\in [a_i,b_i]$ for some $i\in I$ and $y\not\in \bigcup_{j\in I} [a_j,b_j]$  then, $J_I(x,z)\in [1-b_i,1-a_i]$  where 
  $J_I(y,z)=J_{KD}(y,z)=\max\{1-y,z\}$. But, because $ x< y$ 
  then $1-y < 1-b_i$ and therefore,  $J_I(y,z)< J_I(x,z)$.
  \item The case  $x\not\in \bigcup_{i\in I}[a_i,b_i]$ and $y\in [a_j,b_j]$ for some $i\neq j\in I$ is analogous to previou one.
  \item If  $x,y\not\in \bigcup_{i\in I}[a_i,b_i]$ then $J_I(y,z)=J_{KD}(y,z)=\max\{1-y,z\}\leq  \max\{1-x,z\}= J_{KD}(x,z)=J_I(x,z)$.
       \end{enumerate}

 \noindent J2) Let $x,y,z\in [0,1]$ such that $y\leq z$. We have two cases:
       \begin{enumerate}
        \item If $x \in [a_i,b_i]$ for some $i\in I$ then, because $J_i$ satisfy J2,  
        $J_I(x,y)=(1-b_i)+(b_i-a_i) ~J_i\left(\dfrac{x-a_i}{b_i-a_i},y\right)
        \leq (1-b_i)+(b_i-a_i) ~J_i\left(\dfrac{x-a_i}{b_i-a_i},z\right)=J_I(x,z)$.    
         \item If  $x\not\in \bigcup_{i\in I}[a_i,b_i]$ then $J_I(x,y)=J_{KD}(x,y)=\max\{1-x,y\}\leq  \max\{1-x,z\}= J_{KD}(x,z)=J_I(x,z)$.
        \end{enumerate}

 \noindent J3) If $a_i=0$ for some $i\in I$ then $J_I(0,0)=(1-b_i)+b_i ~J_i\left(\dfrac{0}{b_i},0\right)=1-b_i+b_i=1$. 
 
 \indent If $a_i\neq 0$ 
          for each $i\in I$ then  $J_I(0,0)=J_{KD}(0,0)=1$.

 \noindent J4) Since $b_i\neq 1$ for each $i\in I$ then   $J_I(1,1)=J_{KD}(1,1)=1$.

 \noindent J5)  Since $b_i\neq 1$ for each $i\in I$ then $J_I(1,0)=J_{KD}(1,0)=0$.
      \end{proof}\\

\begin{theorem}
  Let $(J_i)_{i\in I}$ be a family of fuzzy implications and $(]a_i, b_i[)_{i\in I}$ be a family of nonempty, pairwise disjoint open 
subintervals of $[0,1]$ such that $b_i< 1$ for each $i\in I$. For each $i\in I$ let $N_i$ the natural negation of the implication $J_i$ and $N_I$ their ordinal sums.
Then, the natural negation of the left ordinal sum of the summands $(a_i, b_i, J_i)_{i\in I}$, denoted by $N_{J_I}$, is such that $N_{J_I}=N_I$. 
\end{theorem}

\begin{proof}
 Since, by Proposition \ref{pro-JI-nova}, $J_I$ is a fuzzy implication then their natural negation  $N_{J_I}$ is in fact a fuzzy negation. On the other hand, 
 by Proposition \ref{PNSum-old}, $N_I$ is also a fuzzy negation. Then $N_{J_I}(0)=N_I(0)$ and $N_{J_I}(1)=N_I(1)$. 
 Let $x\in (0,1)$. If $x\in [a_i,b_i]$ for some $i\in I$ then $N_{J_I}(x)=J_I(x,0)=(1-b_i)+(b_i-a_i) ~J_i\left(\dfrac{x-a_i}{b_i-a_i},0\right)=(1-b_i)+(b_i-a_i)
 ~N_i\left(\dfrac{x-a_i}{b_i-a_i}\right)=N_I(x)$. \\
 Now, if $x\not\in  [a_i,b_i]$ for each $i\in I$ then $N_{J_I}(x)=J_I(x,0)=J_{KD}(x,0)=1-x=N_I(x)$. 
\end{proof}\\

\section{Final Remarks}

Ordinal sum is an important method that allows us to construct, from a family of operators of a certain class, a new operator of the same class. For instance, 
in \cite{SS63}, it was proved that the ordinal sums of a family of triangular norms and triangular conorms result in triangular norms and triangular conorms, 
respectively. This same idea was used for copulas \cite{Nel06}, 
overlaps functions \cite{DB14}, uninorms \cite{MeZ16,MeZ17}, fuzzy implications \cite{DK16,Su15} and fuzzy negations \cite{BSJIFS}.

In \cite{BSJIFS}, it was proved that ordinal sums of a family of fuzzy negations result in a fuzzy negation. In this paper, we studied the ordinal sums of 
some known classes of fuzzy negations and some properties of fuzzy negations, as well as, we studied the ordinal sums of a family of t-norms, t-conorms and fuzzy implications. In particular, we introduce a new way to make an ordinal sum of fuzzy implications which only consider the first variable to decide the case to be considered. The advantage of this new ordinal sums of fuzzy implication, when compared with the several proposes existent in the literature, is that the natural negation is not trivial and commute with the ordinal sums of the natural negations of the summands in the following sense 
\begin{eqnarray*}
\xymatrix{
(a_i,b_i,J_i)_{i\in I} \ar@{->}[rr]^{\hspace{0.5cm} Prop.~ \ref{pro-JI-nova}} \ar@{->}[d]_{Def. ~\ref{DNJ}} & & J_I  \ar@{->}[d]^{Def. ~\ref{DNJ}}\\
(a_i,b_i,N_{J_i})_{i\in I}  \ar@{->}[rr]_{\hspace{0.5cm} Prop.~\ref{PNSum-old}} & & N_{J_I}\\
}
\end{eqnarray*}

Several extensions or types of fuzzy set theory had been proposed in order to solve the problem of constructing the membership degree functions of fuzzy 
sets or/and to represent the uncertainty associated to the considered problem in a way different from fuzzy set theory \cite{bustince15}. As further works, 
we intend to extend the study of ordinal sums of fuzzy negations for some of the more important extensions of fuzzy set theory, such as, ordinal sums of $n$-dimensional 
fuzzy negations \cite{BM18}, ordinal sums of intuitionistic fuzzy negations \cite{costa},  ordinal sums of interval-valued fuzzy 
negations \cite{bedregal10}, 
ordinal sums of interval-valued Atanassov's intuitionistic fuzzy negations \cite{RB17} and 
ordinal sums of typical hesitant fuzzy negations \cite{BSB14}.

\section*{Acknowledgment}

This work is supported by Brazilian National Council of Technological and Scientific Development CNPq (Proc. 307781/2016-0 and 404382/2016-9).

\end{document}